\documentclass[a4paper,12pt,leqno]{article}

\usepackage{amssymb,amsmath}
\usepackage[abbrev]{amsrefs}
\usepackage{amsthm}
\usepackage{hyperref}
\usepackage{xy}
\usepackage{color}

\xyoption{all}

\numberwithin{equation}{section}

\newtheorem{defn}{Definition}[section]
\newtheorem{corollary}[defn]{Corollary}
\newtheorem{rem}[defn]{Remark}
\newtheorem{exm}[defn]{Example}
\newtheorem{lemma}[defn]{Lemma}
\newtheorem{theorem}[defn]{Theorem}

\newcommand\hide[1]{}
\newcommand\V{\bigvee}
\newcommand\ie{i.e.}

\newcommand\st{\mid}
\newcommand\cf{\textrm{cf.}}
\newcommand\opens{\operatorname{\mathcal{O}}}

\newcommand\groupoid{\operatorname{\mathcal{G}}}
\newcommand\spp{\varsigma}

\newcommand\Frm{\textit{Frm}}
\newcommand\Loc{\textit{Loc}}

\newcommand\opp[1]{{#1}^{\textrm{op}}}

\newcommand\rs{\mathrm{R}}
\newcommand\ls{\mathrm{L}}
\newcommand\ts{\operatorname{T}}
\newcommand\ident{\mathrm{id}}
\newcommand\ipi{\mathcal I}

\newcommand\SL{\mathit{SL}}

\newcommand\biIQLoc{\mathbf{IQLoc}}
\newcommand\grpd{\mathit{Gpd}}
\newcommand\HSG{\mathit{Gpd}_{\mathrm{HS}}}

\newcommand\bigrpd{\mathbf{Gpd}}

\newcommand\tspp{\tau}

\newcommand\act{\mathfrak a}
\newcommand\bct{\mathfrak b}
\newcommand\intg[1]{\widetilde{#1}}
\newcommand\lres[2]{_{#1}\vert{#2}}
\newcommand\rres[2]{{#1}\vert_{#2}}
\newcommand\cov{\widehat}

\begin{document}

\title{Principal bundles of open groupoids}

\author{Juan Pablo Quijano}

\maketitle

\maketitle

\vspace*{-1cm}
\begin{abstract}
Taking into account the correspondence between open groupoids and their quantales, we establish a bijective correspondence between the principal $G$-bundles whose left projection is an open surjection and the principal $\opens(G)$-locales for any open groupoid $G$ and its groupoid quantale $\opens(G)$, the latter requires a generalization of the supported modules for groupoid quantales.
\\
\vspace*{-2mm}~\\
\textit{Keywords:} Open groupoids, groupoid quantales, principal bundles, Hilsum--Skandalis maps, supported modules\\
\vspace*{-2mm}~\\
2010 \textit{Mathematics Subject
Classification}: 06D22, 06F07, 18B99, 18F15, 55R10
\end{abstract}

{\small{\tableofcontents}}

\newpage

\section{Introduction}\label{introduction}
Open localic groupoids are in bijective correspondence with groupoid quantales \cite{PR12,QRnonuspp}, which generalizes the one between \'etale localic groupoid and inverse quantal frames of \cite{Re07}. 
The latter is not functorial in the usual sense (groupoid functors do not correspond with homomorphism of unital involutive quantales). Although, since \'etale groupoids have a rich inverse semigroup of local bisections, this correspondence can be extended to a quantale module theoretic description of actions and sheaves \cite{GSQS}, and moreover to a bicategorical bi-equivalence: the bicategory $\bigrpd$, of localic \'etale groupoids with bi-actions as 1-cells, is bi-equivalent to the bicategory $\biIQLoc$, whose objects are the inverse quantal frames and whose 1-cells are quantale bimodules that satisfy a mild condition \cite{Re15}. Furthermore in \cite{Funct2} more general functors between groupoids are studied by using the quantale module theoretical language, namely, good definitions of principal bundle and Hilsum--Skandalis map for inverse quantal frames in those cases where the modules are sheaves are provided. The restriction to sheaves makes no difference for Morita equivalence because the ``Morita bimodules'' for \'etale groupoids are always necessarily sheaves, this leads to a description of Morita equivalence for \'etale groupoids using the Hilbert module language which resembles the imprimitiviy bimodules of Morita equivalence of C*-algebras. This accomplished the previously ongoing program concerning quantales of \'etale groupoids started in \cite{Re07}. 

Despite the fact that \'etale groupoids are useful, and indeed in some cases, groupoids arising in geometric situations (for instance foliations) are Morita equivalent to \'etale groupoids, it is useful to have a better understanding of the relation between quantales and non-\'etale groupoids. In \cite{QuijanoPhD}, the author provided steps for addressing a similar program for quantales of non-\'etale groupoids, in this case studying actions and sheaves for a suitable subclass of open groupoids, namely those with ``\'etale covers" so-called \'etale-covered groupoids: an \'etale-covered groupoid consists of an open (localic) groupoid $G$ which, similarly to \'etale groupoids, is equipped with a well behaved inverse semigroup of local bisections due to which there is an \'etale groupoid $\cov G$ and a surjective groupoid morphism $J:\cov G\to G$ which restricts to an isomorphism $\cov G_0\cong G_0$. We say that $\cov G$ is an ``\'etale cover'' of $G$.  Coverability is not too strong a condition, in the sense that at least it is satisfied by the typical examples of groupoids that arise in analysis and geometry, such as the locally compact groupoids of \cite{Paterson} (provided there are enough local bisections) and, in particular, Lie groupoids. See~\cite{PR12}. From an \'etale-covered groupoid $G$ one therefore obtains two quantales, $\opens(G)$ and $\opens(\cov G)$ (the latter being an inverse quantal frame), and a strong form of embedding $\opens(G)\to\opens(\cov G)$ such that in particular $\opens(G)$ can be regarded as a sub-frame and an ideal in $\opens(\cov G)$. Such a quantale $\opens(G)$ is called an \emph{inverse-embedded quantal frame}. Conversely, from an inverse-embedded quantal frame $\opens$ one obtains an \'etale-covered groupoid, and the two constructions give us a bijection (up to isomorphisms) between the classes of \'etale-covered groupoids and inverse-embedded quantal frames. The latter correspondence and an appropriate notion of action for such quantales yield an equivalence of categories $G$-$\Loc\cong \opens$-$\Loc$ where $\opens=\opens(G)$ is the quantale of an \'etale-covered groupoid $G$. By applying the latter equivalence we characterized completely the sheaves on such groupoids. Moreover, it can be proved that the bicategory whose 1-cells are the \'etale-covered groupoid bi-actions is bi-equivalent to a corresponding bicategory of inverse-embedded quantale frames and bimodules. Although, the Morita theory for \'etale-covered groupoids via their quantales is lacking.

Due to the absence of local bisections, there is no module theoretic description of actions, sheaves, etc, for open groupoids in general so far. However, if we restrict the category of $G$-locales (groupoid actions) and equivariant maps $G$-\Loc\ to a subcategory of groupoid actions so-called  \emph{fully open principal $G$-bundles}: this is a principal $G$-bundle whose left projection is an open surjection, we can establish a bijective correspondence between the fully open principal $G$-bundles and principal $\opens(G)$-locales for any open groupoid $G$. In order to obtain the right notion of principal $\opens(G)$-locale, we need to generalize the notion of stably supported module of \cite{GSQS}. One of the purposes is to recover the algebraic simplicity and convenience of supported modules in the general setting of groupoid quantales. Here, we shall work with a Hilbert $Q$-module  $X$ for a groupoid quantale $Q$, which are $Q$-modules equipped with an inner product satisfying conditions that generalize those of \cite{GSQS}. Then a support $\spp_X:X\to A$ (where $A$ is the base locale of the groupoid quantale $Q$) is defined to be a sup-lattice homomorphism whose axioms interplay with the inner product of the Hilbert structure. So we shall see that notion of equivariantly supported $A$-$A$-quantale will play a central role to extend the notion of stably supported $Q$-module, which preserves all the nice properties of the \'etale case. The latter will be used extensively on the rest of the paper. Finally, we would like to remark that in order to study Morita equivalence for open groupoids the restriction to principal $G$-bundles whose left projection is an open projection makes no difference because the ``Morita bibundles'' have necessarily open surjections as projections. This provides the first steps to establish Morita equivalence for open groupoids in terms of their quantales, in particular, this could be useful for \'etale-covered groupoids where the projections cannot be assumed to be sheaves \cite{QuijanoPhD}. However, the latter will not be addressed in this paper.

This paper is organized as follows: In Section~\ref{section:preliminaries} we fix terminology and give the basic definitions and properties of locales, groupoids and quantales, and finally (bi-)actions of groupoids; in Section~\ref{section:principalbundles} we provide an overview, for open groupoids, of the concept being addressed here, to some extent with the purpose of adapting the presentation of recall the definitions of principal $G$-bundle and Hilsum--Skandalis maps for an open groupoid $G$, which  of~\cites{Mr99,MrcunPhD, Funct2} to the setting of open localic groupoids rather than \'etale groupoids (localic or topological). We shall study some properties of pullbacks of principal bundles in order to prove that the composition of principal bundles is again a principal bundle, therefore, we can define the category $\HSG$ of Hilsum--Skandalis maps; In Section~\ref{section:supportedmodules} we generalize to open groupoids the algebraic theory of supported modules for \'etale groupoids \cite{GSQS}. Whereby, the main aim is not only to complete the toolbox needed in order to establish the main result of this paper but also to develop the theory stably supported modules on its own; in Section~\ref{section:principallocales} we introduce the notion of $Q$-locale for an arbitrary groupoid quantale $Q$ this leads to Theorem~\ref{theo:fullopenlocale}, in which, we shall prove that given a $Q$-locale $X$ we can obtain a $G$-locale whose left projection is an open surjection where $G=\groupoid(G)$. In order to obtain the converse, we have to restrict the category $G$-\Loc\ to a subcategory of fully open principal $G$-bundles, and then, by introducing the notion of principal $Q$-locale, we will obtain our main result Corollary~\ref{cor:bijectionprincipalbundles}.


\section{Preliminaries}\label{section:preliminaries}
We shall use common terminology and notation for  locales,
quantales, and groupoids, following \cite{Funct2,GSQS,Re15}.
\paragraph{Frames and locales.} 
A \emph{frame} is a sup-lattice $L$ satisfying the following distributive law:
\begin{equation}\label{preliminaries, eq: frames}
a \wedge \V B = \V \{ a\wedge b\ \vert\ b\in B  \}\,.
\end{equation}
A \emph{frame homomorphism} is a sup-lattice homomorphism $f:L\to M$ between frames that also preserves binary meets.
The category of frames, formed by frames as objects and frame homomorphisms as morphisms, will be denoted by \Frm.

We denote by $\Loc$ the category of \emph{locales}, which is the dual of $\Frm$. We shall not make any distinction between a locale $X$ regarded as an object of $\Loc$ or as an object of $\Frm = \opp\Loc$ in this paper

If $f:X\to Y$ is a map of locales we shall refer to the frame homomorphism $f^*:Y\to X$ that defines it as its \emph{inverse image homomorphism}. 

A map of locales $f: L\to M$ is said to be \emph{semiopen} if $f^*:M\to L$ preserves all meets (or, equivalently if it has a left adjoint $f_!:X\to Y$ called the \emph{direct image}), and, it is said to be \emph{open} if it is semiopen and satisfies the so-called \emph{the Frobenius condition}:
\begin{equation}\label{preliminaries, eq: frobeniuscondition}
f_!(a\wedge f^*(b)) = f_!(a)\wedge b\,,
\end{equation}
for all $a\in L$ and $b\in M$. The following two lemmas are fundamental properties of open maps of locales:
\begin{lemma}{\cite[section VI.4]{JT}}\label{lemma:BCcondition}
Consider a pullback in \Loc:
\[
\xymatrix{
A \ar[d]_{g}\ar[r]^{h} & B \ar[d]^{f} \ar[d]^{f}\\
C \ar[r]^{j} & D
}
\]
Then if $f$ is open then $g$ is also open. Furthermore we have \emph{the Beck--Chevalley condition}
\[  j^*\circ f_! = g_! \circ h^*\;. \]
\end{lemma}

\begin{lemma}\label{lemma:reflextionopeness}
Consider a pullback in \Loc:
\[
\xymatrix{
A \ar@{->>}[d]^g\ar[r]^{h} & B \ar[d]^{f} \ar@{->>}[d]^f\\
C \ar[r]^{j} & D
}
\]
such that $h$ is open and $f,g$ are open surjections. Then $j$ is open. 
\end{lemma}
\begin{proof}
Let us prove that the direct image of $j$ is given by $\phi:=f_!\circ h_!\circ g^*$. The unit of the adjuction follows from the derivation
\begin{align*}
j^*\circ \phi &= j^*\circ (f_!\circ h_!\circ g^*)=(j^*\circ f_!)\circ(h_!\circ g^*)\\
&=g_!\circ h^*\circ h_!\circ g^*&\text{(by Lemma~\ref{lemma:BCcondition})}\\
&\geq g_!\circ g^*&\text{($h^*\circ h_!\geq \ident$)}\\
&= \ident\;,&\text{($g$ is surjective)}
\end{align*}
and, for the co-unit, we have
\begin{align*}
\phi\circ j^*&=(f_!\circ h_!\circ g^*)\circ j^*=(f_!\circ h_!)\circ(g^*\circ j^*)\\
&=(f_!\circ h_!)\circ(h^*\circ f^*)& \text{($g^*\circ j^*=h^*\circ f^*$)}\\
&\le f_!\circ f^*&\text{($h_!\circ h^*\le \ident$)}\\
&= \ident\;.&\text{($f$ is surjective)}
\end{align*}
Hence, $j_!=\phi$. Finally the Frobenius condition follows, for all $c\in C$ and $d\in D$, from:
\begin{align*}
j_!(c\wedge j^*(d))&=(f_!\circ h_!\circ g^*)(c\wedge j^*(d))\\
&=(f_!\circ h_!)(g^*(c)\wedge g^*(j^*(d)))\\\
&=(f_!\circ h_!)(g^*(c)\wedge h^*(f^*(d)))& \text{($g^*\circ j^*=h^*\circ f^*$)}\\
&=f_!(h_!(g^*(c))\wedge f^*(d))& \text{[by \eqref{preliminaries, eq: frobeniuscondition} of $h$]}\\
&=f_!(h_!(g^*(c)))\wedge d& \text{[by \eqref{preliminaries, eq: frobeniuscondition} of $f$]}\\
&=j_!(c)\wedge d\;.
\end{align*}

\end{proof}

The product of $X$ and $Y$ in $\Loc$ is $X\times Y$, and it coincides with the tensor product $X\otimes Y$ in $\SL$. If $h_1,h_2: L\to M$ are frame homomorphisms consider
\[  E=E(h_1,h_2)=\{  x\in L \st\ h_1(x)=h_2(x) \}\;.  \]
Obviously $E$ is a subframe of $L$ and the embedding $j:E\to L$ has the property that whenever $h_1\circ g=h_2\circ g$ for a frame homomorphism $g: N\to L$ then we have (precisely one) frame homomorphism $\bar{g}: N\to E$ with $j\bar{g}=g$ (namely, the one given by $\bar{g}(x)=g(x)$).
Consequently, the category \Loc\ has coequalizers: if $f_1.f_2: M\to L$ are localic maps, the coequalizer can be obtained as 
\[   \pi: L\to L/M    \]
given by $\pi(x)=\upsilon_{L/M}(x)$ where $\upsilon$ is the corresponding nucleus (see, \cite[III.4.3]{picadopultr}).

To conclude this paragraph we prove the following fact of pullbacks of co-equalizer of locales:
\begin{lemma}\label{lem:isomreflex}
Let the following be a pullback diagram in $\Loc$ where $\pi_1$ is an isomorphism:  
\begin{equation}\label{eq:isomreflex}
\xymatrix{
X\times_B Y\ar@{->>}[rr]^-{\pi_2}\ar[d]^{\cong}_{\pi_1}&&Y\ar[d]^p\\
X\ar[rr]_-q&& B
}
\end{equation}
and $\pi_2$ is an open surjection and $q$ is a co-equalizer. Then $p$ is an isomorphism.
\end{lemma}
\begin{proof}
Let $A$ be a locale and  $u_1,u_2: A\to X$ be two map of locales, let us define a map of locales $y:=\pi_2\circ\pi^{*}_1\circ u_2$ which exists because $\pi_1$ is an isomorphism by assumption. Taking into account that $q$ is a co-equalizer we have
\[ p\circ y=p\circ \pi_2\circ \pi^{*}_1\circ u_2=q\circ \pi_1\circ \pi^*_1\circ u_2=q\circ u_2=q\circ u_1\;. \]
Since \eqref{eq:isomreflex} is a pullback, there exists a unique map of locales $b:A\to X\times_B Y$ such that $\pi_1\circ b=u_1$ and $\pi_2\circ b=y$. The former equation implies that $b=\pi^{*}_1\circ u_1$, combining this with the latter equation, we have
\[ \pi_2\circ \pi^{*}_1\circ u_1=y=\pi_2\circ\pi^{*}_1\circ u_2\;,  \] 
and since $q$ is the co-equalizer of $u_1$ and $u_2$, there is a unique map of locales  $c: B\to Y$ with $c\circ q=\pi_2\circ \pi^{*}_1$. Thus $p\circ c\circ q=q$, which implies that $p\circ c=\ident$ because $q$ is surjective. Furthermore, $c\circ p\circ \pi_2=c\circ q\circ \pi_1=\pi_2$ implies $c\circ p=\ident$ because $\pi_2$ is surjective. Therefore $p$ is an isomorphism.

\end{proof}


\paragraph{Groupoids and quantales.} 
A \emph{localic groupoid} is an internal groupoid in the category of locales $\Loc$. We denote the locales of objects and arrows of a groupoid $G$ respectively by $G_0$ and $G_1$,
\[
\xymatrix{
G=\quad G_2\ar[r]^-m&G_1\ar@(ru,lu)[]_i\ar@<1.2ex>[rr]^r\ar@<-1.2ex>[rr]_d&&G_0\ar[ll]|u
}
\]
where $G_2$ is the pullback of the $d$ and $r$ in $\Loc$. A localic groupoid is \emph{open} if $d$ is an open map, in which case, the \emph{range map} $r$ and the \emph{multiplication map} $m$ are open, as well. An \emph{\'etale groupoid} is an open groupoid such that $d$ is a local homeomorphism, in which case all the structure maps are local homeomorphisms. We shall denote the quantale of an open groupoid $G$ by $\opens(G)$, and describe it briefly, following the construction of \cite{QRnonuspp, QuijanoPhD}: 
Let $A$ be a locale. By an  \emph{$A$-$A$-bimodule} $M$ is meant a sup-lattice equipped with two unital (resp. left and right) $A$-module structures
\begin{equation}\label{AAbimodule} 
(a,m)\mapsto {\lres{a}{m}}\quad \text{and}\quad (a,m)\mapsto \rres{m}{a}\,,
\end{equation} 
satisfying the following additional condition, for all $a, b \in A$ and $m \in M$:
\begin{equation}\label{def:AAbimodules}
  \rres{(\lres{a}{m})}{b} = {\lres{a}{(\rres{m}{b})}}\,.  
\end{equation}
An \emph{$A$-$A$-quantale} $Q$ is just a semigroup in the monoidal category of $A$-$A$-bimodules, \ie\ 
$Q$ is an $A$-$A$-bimodule equipped with a quantale multiplication $(x, y)\mapsto xy$ satisfying the following additional conditions, for all $a \in A$ and $x, y \in Q$:
\begin{equation}\label{AAquantale}
 (\lres{a}{x})y=  {\lres{a}{(xy)}}\;,\ (\rres{x}{a})y =   x(\lres{a}{y})\;,\ \rres{(xy)}{a} =   x(\rres{y}{a})\,.  
\end{equation}
An $A$-$A$-quantale $Q$ is \emph{involutive} if it is an involutive semigroup, that is, $Q$ is endowed with a self-adjoint map (called the \emph{involution}) given by $a\mapsto a^{*}$ and it is required to satisfy, besides the standard conditions $x^{**}=x$ and $(xy)^{*}=y^{*}x^{*}$, the following two conditions:
\begin{equation}\label{involution}
 (\V_i x_i)^{*} =\V_i x^{*}_i\,,\ (\lres{a}{\rres{x}{b}})^{*} = {\lres{b}{\rres{x^{*}}{a}}}\;. 
\end{equation}
By a \emph{based quantale} will be meant an involutive quantale $Q$ equipped with the structure of an involutive $A$-$A$-quantale for some locale $A$.
A based quantale $Q$ is \emph{supported} if it is equipped with a sup-lattice homomorphism $\spp_Q: Q\rightarrow A$ satisfying the following conditions for all $x,y\in Q$:
\begin{equation}\label{sppq}
\spp_Q(1_Q)=1_A\;,\  \lres {\spp_Q(x)}{y} \leq xx^{*}y\,,\ \lres {\spp_Q(x)}{x}= x\,. 
\end{equation}

Let us denote by $\rs(Q)$ the set of \emph{right-sided} elements of $Q$, an element $q\in Q$ is right-sided if $q1_Q\leq q$. Similarly, the set of \emph{left-sided} elements of $Q$ is denoted by $\ls(Q)$ and the set of \emph{two-sided} elements is denoted by $\ts(Q)$. 
A support $\spp_Q$ is said to be \emph{equivariant} if 
\begin{equation}
\spp_Q(\lres{a}{x})=a\wedge \spp_Q(x)\;.
\end{equation}
Then, the mapping $A\to \rs(Q)$ defined by $x\mapsto {\lres{x}{1_Q}}$ is an order isomorphism whose inverse is the mapping $\rs(Q) \to A$ defined by $x\mapsto \spp(x)$. In particular $\rs(Q)\cong A$, i.e., $\rs(Q)$ is a locale. 
Any equivariant support is necessarily \emph{stable}, by which it is meant that the following equivalent conditions hold for all $x,y\in Q$;
\begin{equation}\label{supportstable}
\spp(xy)\le \spp(x)\;,\  \spp(x1_Q)\le\spp(x)\;,\ \spp(xy) = \spp(\rres{x}{\spp(y)})\;.
\end{equation}
A based quantale is \emph{stably supported} if it is equipped with a stable support.

By a \emph{based quantal frame} is meant a based quantale $Q$ such that for all $q,m_i\in Q$ and $a\in A$ the following properties hold:
\begin{equation}\label{qf3}
q \wedge \bigvee_i m_i = \bigvee_i q\wedge m_i\;,\ (\lres{a}{q})\wedge m = {\lres{a}{(q\wedge m)}}\;,\ m\wedge (\rres{q}{a}) = \rres{(q\wedge m)}{a}\;. 
\end{equation}
By a \emph{reflexive quantal frame} $(Q,\upsilon)$ will be meant an $A$-$A$-quantal frame equipped with a frame homomorphism $\upsilon: Q\to A$ satisfying, for all $a\in A$:
\begin{equation}\label{eq:reflexivity}
\upsilon(\lres{a}{1_Q})=a=\upsilon(\rres{1_Q}{a})\,.
\end{equation} 
The quantale multiplication has the following factorization in $\SL$
\[
\vcenter{\xymatrix{
 Q\otimes Q \ar@{->>}[d]_{\pi}  \ar[dr]^{\mu} \\
 Q\otimes_{A} Q  \ar[r]_-{\mu_A} & Q\,.
}
}
\]
By a \emph{multiplicative quantal frame} is meant an $A$-$A$-quantale frame such that the right adjoint of $\mu_A$ preserves joins. 
Let $(Q,\spp,\upsilon)$ be a multiplicative equivariantly supported reflexive quantal frame. We say that $Q$  satifies \emph{unit laws} if moreover the following condition holds for all $a\in Q$:
\begin{equation}\label{unitlaws}
\bigvee_{xy\leq a} (\lres{\upsilon(x)}{y}) = a\,.
\end{equation} By a \emph{groupoid quantale} $Q$ will be meant a multiplicative equivariantly supported reflexive quantal frame that satisfies unit laws and moreover satisfies the following condition, which is referred to as the \emph{inverse laws}, for all $a,x\in Q$:
\begin{equation}\label{inverselaws}
\lres{\upsilon(a)}{1_Q} = \bigvee_{xx^*\leq a} x\,.
\end{equation}  
The groupoid quantales $Q$ are precisely the quantales of the form $Q\cong \opens(G)$ for an open groupoid $G$ and the class of unital equivariantly supported reflexive quantal frames satisfying \emph{inverse laws} corresponds with the class of \emph{inverse quantal frames}. 
Recall that by an inverse quantal frame $Q$ is meant an unital stable quantal frame, satisfying
\begin{equation}
\bigvee \ipi(Q)= 1_Q\,,
\end{equation}
where $\ipi(Q) = \{a \in Q\ \vert\ a^*a\vee a^*a\leq e\}$ (partial units of $Q$). The inverse quantal frames $Q$ are exactly the quantales of the form $Q = \mathcal{O}(G)$ for an \'etale groupoid $G$  \cite[Theo. 5.11]{Re07}.

\paragraph{Actions of open groupoids.}
Let $G$ be an open groupoid. A \emph{(left) $G$-locale} is a triple $(X,\act,p)$ where $X$ is a locale together with a map $p: X\to G_0$ (called \emph{achor map or projection map}) and a map of locales (called the \emph{action}) $\act: G_1\times_{G_0} X \to X$ such that the following diagrams commute:
\begin{description}
\item{ \textbf{Pullback}:}\label{defactiongpd1}
\[
\vcenter{\xymatrix{
 G_1\times_{G_0} X  \ar[d]_{\pi_1} \ar[r]^-{\act} & X \ar[d]^{p}  \\
 G_1  \ar[r]_-{d} & G_0 }
}
\]
\item{\textbf{Associativity}:}
\[
\vcenter{\xymatrix{
 G_1\times_{G_0} (G_1\times_{G_0} X) \ar[d]_{\cong} \ar[r]^-{\ident_{G_1}\times \act}  & G_1\times_{G_0} X \ar[dd]^{\act} \\
 G_2\times_{G_0} X  \ar[d]_{m\times \ident_X} \\
 G_1\times_{G_0} X  \ar[r]_-{\act} & X }
}\label{defactiongpd2}
\]
\item{\textbf{Unitary}:}
\[
\vcenter{\xymatrix{
 & G_1\times_{G_0} X \ar[dr]^{\act} \\
 X \ar[ur]^{\langle u\circ p, \ident_X \rangle} \ar[rr]_{ \cong}  && X }
}\label{defactiongpd3}
\]
\end{description}
Let $(X, \act_X, p_X)$ and $(Y, \act_Y, p_Y)$ be $G$-locales or simply $X$ and $Y$ when no confusion arise, for any 
$G$-locale $(X,\act, p)$, the action $\act$ is necessarily open because the diagram~\textbf{Pullback} is, in fact, a pullback and $d$ is an open map. By an \emph{open $G$-locale} is meant a $G$-locale $(X,\act,p)$ such that $p$ is an open map of locales. Moreover, if $p$ is a surjective map of locales, we say that $X$ is a \emph{fully open $G$-locale}. An \emph{equivariant map} from $X$ to $Y$ is a map $f:X\to Y$ in the slice category $\Loc/G_0$ that commutes with the actions, 
\begin{equation}\label{diag:equivariantmap} 
\xymatrix{
G_1\times_{G_0}X \ar_{\act}[d] \ar^-{\ident_{G_1}\times f}[rr]&& G_1\times_{G_0}Y \ar^{\bct}[d]\\
X \ar_{p_X}[dr] \ar^{f}[rr] && Y\ar^{p_Y}[dl]\\
& G_0
}
\end{equation}
we shall call to the category of $G$-locales and equivariant maps between them as \emph{$G$-\Loc}. Furthermore, The categories of left and right $G$-\Loc\ are isomorphic (see, \cite[Lemma 2.1]{Funct2}). Similarly, we define  \emph{\textit{open} $G$-\Loc}\ and \emph{\textit{fully open} $G$-\Loc}, the categories of open $G$-locales and fully open $G$-locales with equivariant maps, respectively.

\paragraph{$\opens(G)$-modules.}
Let $G$ be an open groupoid. We shall denote by $X$ the $\opens(G)$-module which is obtained from a $G$-locale $X$, as follows; taking into account that $G_1\times_{G_0} X$ is a locale, a quotient  $G_1\otimes_{G_0} X$ of the tensor product $G_1\otimes X$, we obtain a sup-lattice homomorphism by composing with the direct image of the action:
\[
\vcenter{\xymatrix{
 G_1\otimes X   \ar@{->>}[r]  & G_1\otimes_{G_0} X \ar[r]^-{\act_!} & X
 }
}
\]
showing that this defines an action of $\opens(G)$ on $X$ (a left quantale module). Hence, if $G$ is an open groupoid, in which case as we have seen before the actions are open maps, the
following diagram in $\SL$ also commutes (cf, \cite{JT} Proposition V.4.1):
\[ \xymatrix{
\opens(G_1)\otimes_{\opens(G_0)}X \ar_{\act_!}[d] && \opens(G_1)\otimes_{\opens(G_0)}Y \ar^-{\ident_{\opens(G_1)}\otimes f^*}[ll] \ar^{\bct_!}[d]\\
X  && Y\ar^-{f^*}[ll]
}
\]
it means that $f^*$ commutes with the actions of $\opens(G)$ on $X$ and $Y$ respectively, \ie, it is a homomorphism of $\opens(G)$-modules.

Furthermore, If $X$ is a $G$-locale with open anchor map $p:X\to G_0$. Then $X$ is an $\opens(G_0)$-module by change of "ring" along the inverse image $p^*:\opens(G_0)\to X$. Considering the tensor product $\opens(G)\otimes_{\opens(G_0)} X$, the associativity of the action $\opens(G)\otimes X\to X$ yields a factorization:
\[ \xymatrix{
\opens(G)\otimes X\ar@{->>}[rr]\ar[drr] && \opens(G)\otimes_{\opens(G_0)} X \ar[d]^{\alpha}\\
&& X\;,
}
\]  
the right adjoint of the sup-lattice homomorphism $\alpha:\opens(G)\otimes_{\opens(G_0)} X\to X$ is given by
\begin{align}\label{eq:rightadjointalpha}
\alpha_*(x) = \V \{ q\otimes y\in \opens(G)\otimes_{\opens(G_0)} X \st \alpha(q,y)=q\cdot y\leq x  \}\;.
\end{align}

\paragraph{Groupoid bi-locales.}
Let $G$ and $H$ be two open groupoids. A \emph{$G$-$H$-bilocale} is a locale $X$, equipped with a left $G$-locale structure $(p, \act)$ and a right $H$-locale structure $(q, \bct)$ such that the following diagrams in $\Loc$ are commutative:
\begin{enumerate}
\item \textbf{$q$ is invariant
under the action of $G$} \[
\vcenter{\xymatrix{
G_1\times_{G_0} X \ar[rr]^-{\act} \ar[d]_{\pi_2} && X \ar[d]_-{q} \\
X \ar[rr]_{q}  && H_0   }
}
\]
\item \textbf{$p$ is invariant
under the action of $H$}\[
\vcenter{\xymatrix{
X\times_{H_0} H_1 \ar[rr]^-{\bct} \ar[d]_{\pi_1} && X \ar[d]_-{p} \\
X \ar[rr]_{p}  && G_0   }
}
\]
\item \textbf{Associativity}\[
\vcenter{\xymatrix{
G_1\times_{G_0} X\times_{H_0} H_1 \ar[rr]^-{\act\times \ident_{H_1}} \ar[d]_{\ident_{G_1}\times b} && X\times_{H_0} H_1 \ar[d]_{\bct} \\
G_1\times_{G_0} X \ar[rr]_-{\act}  && X   }
}
\]
\end{enumerate}
we shall often use the notation ${_G X_H}$ to indicate that $X$ is a $G$-$H$-bilocale.
A \emph{map of bilocales} $f: {_G X_H}\to  {_G Y_H}$ is a map of locales that is both a map of the left $G$-locales and a map of right $H$-locales. We will denote by $G$-$H$-$\Loc$ the category of bilocales. The maps of bilocales are the $2$-cells of a bicategory, denoted by $\grpd$,  whose 0-cells are the open groupoids and whose 1-cells are the $G$-$H$-bilocales. The composition of 1-cells is defined bu the tensor product: given 1-cells $_G X_H$ and $_H Y_K$ we define $Y\circ X=X\otimes_H Y$ defined by the coequalizer of
\begin{equation}\label{eq: tensorproductgloc}
\xymatrix{
X\times_{H_0} H_1\times_{H_0} Y\ar@<0.7ex>[rr]^-{\langle a\circ \pi_{12}, \pi_3\rangle}\ar@<-0.7ex>[rr]_-{\langle \pi_1,b\circ \pi_{23}\rangle}&& X\times_{H_0} Y \ar@{->>}[r]& X\otimes_{H} Y\,.
}
\end{equation}


\section{Principal $G$-bundles}\label{section:principalbundles}

Let $G$ be an open groupoid and $M$ be a locale. A \emph{(left) $G$-bundle over $M$} is a (left) $G$-locale $(X, \act, p)$ endowed with a map of locales $\pi: X\to M$, such that the following diagram:
\begin{equation}\label{diag:invarianceofaction}
\vcenter{\xymatrix{
G_1\times_{G_0} X\ar[r]^-{\act} \ar[d]_{\pi_2}  &X\ar[d]^{\pi}\\
X\ar[r]_{\pi} & M\;.
}
}
\end{equation}
is commutative in \Loc. A $G$-bundle over $M$ is called a \emph{principal $G$-bundle over $M$} if the \emph{pairing map}
\begin{equation}\label{principaliso}
\langle \act, \pi_2 \rangle: G_1\times_{G_0} X \to X\times_{M} X
\end{equation}
is an isomorphism of locales and the map $\pi:X\to M$ is an open surjection. Similarly, we define (right) $G$-bundle and a principal (right) $G$-bundle (with pullback over the range map $r$) over $M$.
Recall that, if $G$ is an open groupoid and $X$ is a (left) $G$-locale, the \emph{orbit locale} of the action $X/G$ can be construct as the coequalizer in \Loc:
\begin{equation}
\xymatrix{
G_1\times_{G_0} X\ar@<0.7ex>[rr]^-{\act}\ar@<-0.7ex>[rr]_-{\pi_2}&& X \ar@{->>}^-{\pi}[r]& X/G.
}
\end{equation}
Moreover, for any $X$ principal $G$-bundle over $M$ we have $M\cong X/G$ (see, \cite[Lemma 4.2]{Funct2}).

\paragraph{Point-set reasoning}
The following provides an example of how topological arguments using points can be translated to localic arguments \cite{Vi07} and it was already introduced in \cite{QuijanoPhD}. Let $G$ be an open groupoid and $X$ a principal $G$-bundle with action $\act$. We shall usually denote the inverse of the isomorphism $\langle \act,\pi_2\rangle$ of \eqref{principaliso} by $\langle \theta,\pi_2\rangle$, 
\[
\xymatrix{
X\times_{M} X\ar@<1.0ex>[rr]_{\cong}^{ \langle \theta,\pi_2 \rangle}&& G_1\times_{G_0} X \ar@<1.0ex>[ll]^{ \langle \act,\pi_2 \rangle}\;,
}
\]
where $\theta:X\times_{M} X\to G_1$ is a map of locales. Note that, if $G$ is a topological groupoid and $X$ a topological principal $G$-bundle, $\theta(x,x')$ is, given some pair $(x,x')\in X\times_M X$, the unique arrow $g\in G_1$ such that $gx'=x$ (where we write $g\cdot x$ instead of $\act(g,x)$ for the action of an arrow $g\in G_1$ on an element $x\in X$). It follows, in particular, that $\theta(x,x)$ is the identity arrow $u(p(x))\in G_1$. Such properties are less easy to describe in the case of localic groupoids, but we obtain similar statements by replacing elements such as $x\in X$ by maps $\boldsymbol x$ into $X$, and expressions such as $\theta(x,x')$ by $\theta\circ\langle\boldsymbol x,\boldsymbol x'\rangle$, as follows:

\begin{lemma}\label{localicreasoningwithglobalpoints}
Let $G$ be an open groupoid and $(X,p,\act,\pi)$ a principal $G$-bundle over $M$. For all locales $Z$ and all maps $\boldsymbol g:Z\to G_1$ and $\boldsymbol x,\boldsymbol x':Z\to X$ such that $\pi\circ\boldsymbol x=\pi\circ\boldsymbol x'$ we have
\begin{enumerate}
\item $\theta\circ \langle \boldsymbol x, \boldsymbol x' \rangle=i\circ \theta\circ \langle \boldsymbol x', \boldsymbol x \rangle$.
\item $\theta\circ \langle \act\circ \langle \boldsymbol g, \boldsymbol x \rangle, \boldsymbol x' \rangle=m\circ \langle  \boldsymbol g, \theta\circ \langle \boldsymbol x, \boldsymbol x' \rangle \rangle$.
\end{enumerate}
\end{lemma}
\begin{proof}
1.  The categories of left $G$-locales and right $G$-locales are isomorphic (see, \cite[Lemma 2.1]{Funct2}). Moreover, by principality, the lower rectangle of the following diagram:
\[\xymatrix{ 
G_1\times_{G_0} X\ar[rr]^{\cong} \ar@/_4pc/[dd]_{\pi_1}   && X\times_{G_0} G_1 \ar@/^4pc/[dd]^{\pi_2}\\
X\times_{M} X\ar[u]^{\langle \theta,\pi_2\rangle}_{\cong} \ar[d]^{\theta}\ar[rr]^{\langle \pi_2,\pi_1 \rangle}_\cong  &&  X\times_{M} X\ar[u]_{\langle \pi_1,\theta\rangle}^{\cong}\ar[d]^{\theta}\\ 
G_1\ar[rr]_{i}^{\cong} && G_1
}
\]
is commutative. Hence
\begin{align*}
\theta\circ \langle \boldsymbol x, \boldsymbol x' \rangle &= \theta\circ \langle \pi_2,\pi_1 \rangle\circ \langle \boldsymbol x', \boldsymbol x \rangle\\
&= i\circ \theta \circ \langle \boldsymbol x', \boldsymbol x \rangle\;.
\end{align*}

2. The associativity law and the principality of $X$  give us the following equality:
\begin{equation}\label{eq:mthetaa}
(m\times \ident_X)\circ (\ident_{G_1}\times \langle \theta,\pi_2 \rangle)\circ \langle \boldsymbol g,\langle \boldsymbol x, \boldsymbol x' \rangle \rangle= \langle \theta,\pi_2 \rangle\circ (\act \times \ident_X)\circ \langle \boldsymbol g,\langle \boldsymbol x, \boldsymbol x' \rangle \rangle\;.
\end{equation}
Then,
\begin{align*}
m\circ \langle \boldsymbol g, \theta \circ \langle \boldsymbol x, \boldsymbol x' \rangle\rangle  &= \pi_1\circ \langle m\circ \langle \boldsymbol g, \theta\circ \langle \boldsymbol x, \boldsymbol x' \rangle \rangle, \boldsymbol x' \rangle\\
&= \pi_1\circ (m\times \ident_X)\circ \langle \boldsymbol g, \langle \theta\circ \langle \boldsymbol x, \boldsymbol x' \rangle, \boldsymbol x' \rangle \rangle\\
 &= \pi_1\circ (m\times \ident_X)\circ (\ident_{G_1}\times \langle \theta,\pi_2 \rangle)\circ \langle \boldsymbol g,\langle \boldsymbol x, \boldsymbol x' \rangle \rangle\\
&= \pi_1\circ \langle \theta,\pi_2 \rangle\circ (\act \times \ident_X)\circ \langle \boldsymbol g,\langle \boldsymbol x, \boldsymbol x' \rangle \rangle\quad \text{[by \eqref{eq:mthetaa}]}\\
&= \pi_1\circ \langle \theta,\pi_2 \rangle\circ \langle \act\circ \langle \boldsymbol g, \boldsymbol x \rangle, \boldsymbol x' \rangle\\
&= \pi_1\circ \langle \theta\circ \langle \act\circ \langle \boldsymbol g, \boldsymbol x \rangle, \boldsymbol x' \rangle,  \boldsymbol x' \rangle\\
&= \theta\circ \langle \act\circ \langle \boldsymbol g,\boldsymbol x \rangle, \boldsymbol x' \rangle\;.
\end{align*}

\end{proof}


\paragraph{Pullbacks of principal $G$-bundles.}

The purpose here is to provide a sequence of results regarding pullbacks of principal bundles of open groupoids, in order to show that the composition of principal bundles is again a principal bundle.

\begin{lemma}\label{preliminaries, lemma:pullbackprincipal}
Let $G$ be an open groupoid, $X$ be a principal $G$-bundle over $M$ and $f: \intg M\to M$ be a map of locales. Then there exists a principal $G$-bundle $\intg X$ over $\intg M$ together with an equivariant map $\intg f: \intg X\to X$ such that
\[
\vcenter{\xymatrix{
\intg X\ar[r]^-{\intg f} \ar[d]_{\intg \pi}  & X \ar@{->>}[d]^{\pi}\\
\intg M\ar[r]_-{f}& M\;.
}
}
\]
Furthermore the following universal property holds: for any principal $G$-bundle $X'$ over $\intg M$ and an equivariant map $f':X'\to X$, there is a unique isomorphism of principal $G$-bundles $\varphi: X'\to \intg X$ with $\intg f\circ \varphi=f'$.
\end{lemma}

\begin{proof}
Let $X$ be a principal $G$-bundle over $M$ and $f: \intg M\to M$ be a map of locales. Let us define $\intg X:= X\times_M \intg M$ as the pullback of the maps of locales $\pi:X\to M$ and $f: \intg M\to M$, respectively, \ie
\[
\vcenter{\xymatrix{
\intg X\ar[r]^-{\pi_2} \ar[d]_{\pi_1}  & \intg M \ar[d]^{f}\\
X\ar@{->>}[r]^-{\pi}& M
}
}
\]
we notice that $\pi_2:\intg X\to \intg M$ is an open surjection because $\pi$ so is (open surjections are stable under pullbacks \cite{JT}), moreover we can endowed $\intg X$ with a $G$-action given by
\[ \vcenter{\xymatrix{
G_1\times_{G_0} \intg X = G_1\times_{G_0} (X\times_M \intg M)\ar[rr]^-{\act\times \ident_{\intg M}} && X\times_M \intg M=\intg X\,,   
}
}
\]
with projection $p\circ \pi_1: \intg X\to G_0$\;. It turns out that $\intg X$ is a principal $G$-bundle over $\intg M$\;. Indeed,
\begin{align*}
G_1\times_{G_0} \intg X &= G_1\times_{G_0} (X\times_M \intg M)\\
&\cong (G_1\times_{G_0} X)\times_M \intg M\\
&\cong (X\times_{M} X)\times_M \intg M\quad \text{[$X$ is a principal $G$-locale over $M$]}\\
&\cong (X\times_M \intg M)\times_{\intg M} (X\times_{M} \intg M)\\
&= \intg X \times_{\intg M} \intg X\,.
\end{align*} 
Finally, notice that
\[
\vcenter{\xymatrix{
\intg X\ar[r]^-{\pi_1:=\intg f} \ar[d]_{\pi_2:=\intg \pi}  & X \ar@{->>}[d]^-{\pi}\\
\intg M \ar[r]^-{f}& M\,,
}
}
\]
is a commutative diagram as required. Now, let $X'$ be a principal $G$-bundle over $\intg M$ with an equivariant map $f':X'\to X$. Let us define the map of locales $\varphi$ uniquely by
\[
\vcenter{\xymatrix{X'\ar@/_/[ddr]_{f'}\ar@/^/[drrr]^{\pi'}\ar@{.>}[dr]|\varphi\\
&\intg X = X\times_{M} \intg M \ar[rr]^-{\pi_2}\ar[d]^{\pi_1}&& \intg M\ar[d]^f\\
&X\ar@{->>}[rr]_-{\pi}&&M}}\hspace*{1cm}\varphi:=\langle f, \pi'\rangle\;.
\]
and let us show that $\varphi$ is an isomorphism. In order to do that, let us consider the following pullback
\[
\vcenter{\xymatrix{
\intg X\times_{\intg M} X'\ar[r]^-{\pi_2} \ar[d]_{\pi_1}  & X' \ar@{->>}[d]^-{\pi'}\\
\intg X \ar[r]^-{\intg \pi}& \intg M\,,
}
}
\]
we notice that 
\[ \intg X\times_{\intg M} X'\cong (X\times_M \intg M)\times_{\intg M} X'\cong X\times_M X'\;, \]
thus, there is an open surjection $\xymatrix{X\times_M X'\ar@{->>}[r]^-{\ident_X\times \pi'}  &\intg X=X\times_M \intg M\;,}$ which gives rise to the following pullback
\[
\vcenter{\xymatrix{
X'\times_{\intg X} (X\times_M X')\ar[r]^-{\pi_{23}} \ar[d]_{\pi_1}  & X\times_M X' \ar@{->>}[d]^-{\ident_X\times \pi'}\\
X' \ar[r]^-{\varphi}& \intg X\,.
}
}
\]
Let us prove that $\pi_{23}$ is an isomorphism. Indeed, there is a unique map of locales $\psi:X'\times_{\intg M} X'\to X'\times_{\intg X} (X\times_M X')$ given by the composition of pullbacks:
\[
\vcenter{\xymatrix{X'\times_{\intg M} X'\ar@/_/[ddr]_{\pi_1}\ar@/^/[drrr]^{\pi_2}\ar@{.>}[dr]|{\exists! \psi}\\
&X'\times_{\intg X} (X\times_M X') \ar[r]_-{\pi_{23}}\ar[d]^{\pi_1}& X\times_M X'\ar@{->>}[d]^-{\ident_X\times \pi'}\ar[r]_-{\pi_2} & X' \ar@{->>}[d]^-{\pi'}\\
&X'\ar@{->>}[r]_-{\varphi}& \intg X \ar[r]_{\pi_2} & \intg M}}
\]
however since $X'$ is a principal $G$-bundle over $\intg M$ and $f': X'\to X$ is an equivariant map (and, in fact, the diagram~\eqref{diag:equivariantmap} is a pullback). Then the following isomorphisms hold
\begin{align*}
G_1\times_{G_0} X' &\cong X'\times_{\intg M} X' \cong X'\times_X (G_1\times_{G_0} X)\\
&\cong X'\times_X (X\times_{M} X)\quad \text{[$X$ is a principal $G$-bundle over $M$]}\\
&\cong X'\times_{M} X\cong X\times_M X'\;.
\end{align*}
Therefore $X\times_M X'\cong X'\times_{\intg M} X'$ which, implies that $\pi_{23}\circ \psi=\ident$ and $\psi\circ \pi_{23}=\ident$\;. Hence $\pi_{23}$ is an isomorphism and so is $\varphi$ by Lemma~\ref{lem:isomreflex}.
 \end{proof}

\begin{lemma}\label{preliminaries, lemma: f/G}
Let $G$ be an open groupoid and $X_1, X_2$ be principal $G$-bundles over $M_1$ and $M_2$, respectively. Then an equivariant map $f:X_1\to X_2$ induces a map $f/G: M_1\to M_2$. Furthermore, $f$ is an (open) isomorphism if and only if $f/G$ is an (open) isomorphism. 
\end{lemma}
\begin{proof}
Since $M_i\cong X_i/G$, the equivariant map $f$ induces a map $f/G: X_1/G\to X_2/G$. Moreover, the following commutative diagram:
\[
\vcenter{\xymatrix{
X_1 \ar[r]^-{f} \ar@{->>}[d]_-{\pi'}  & X_2 \ar@{->>}[d]^-{\pi''}\\
X_1/G \ar[r]_-{f/G}& X_2/G
}
}
\]
is a pullback in \Loc\ by Lemma~\ref{preliminaries, lemma:pullbackprincipal}. Furthermore, $f$ is an (open) isomorphism if and only if $f/G$ is an (open) isomorphism, is due to (Lemmas~\ref{lemma:BCcondition} and \ref{lemma:reflextionopeness}) Lemma~\ref{lem:isomreflex}, respectively.
\end{proof}

\begin{lemma}\label{preliminaries, lemma: twopullbacks}
Let $G$ be an open groupoid and $X_i$ be principal $G$-bundles over $M_i$ for all $i=1,2,3,4$. Consider a commuting square of equivariant maps between principal $G$-bundles $X_i$ and the corresponding mappings between their base spaces $M_i$:
\[
\vcenter{\xymatrix{
X_1\ar[r]^-{\gamma_2} \ar[d]_{\gamma_3}  & X_2 \ar[d]^{\delta_2}\\
X_3\ar[r]_-{\delta_3}& X_4\;
}
}
\quad\quad\quad\quad
\vcenter{\xymatrix{
M_1\ar[r]^-{\beta_2} \ar[d]_{\beta_3}  & M_2 \ar[d]^{\alpha_2}\\
M_3\ar[r]_-{\alpha_3}& M_4\,.
}
}
\]
The leftmost square is a pullback in \Loc\ if and only if the rightmost one is.
\end{lemma}
\begin{proof}
Let us suppose that the leftmost square is a pullback, \ie\ $X_1\cong X_2\times_{X_4} X_3$\;, and let $Z$ be a locale and $z_i:Z\to M_i$ be maps of locales with $\alpha_2\circ z_2=\alpha_3\circ z_3:=z_4$ for all $i:2,3,4$; diagrammatically, we have
\begin{equation}\label{diag:pullbackorbitspaces}
\xymatrix{
Z \ar@/_/[ddr]_{z_3}\ar@/^/[drr]^{z_2}\\
&M_1\ar[r]^-{\beta_2} \ar[d]_{\beta_3}  & M_2 \ar[d]^{\alpha_2}\\
&M_3\ar[r]_-{\alpha_3}& M_4 \;.
}
\end{equation}
The equivariant maps $\delta_2$ and $\delta_3$ yield the following pullbacks, by Lemma~\ref{preliminaries, lemma:pullbackprincipal},
\[
\xymatrix{
X_i\cong X_4\times_{M_4} M_i\ar[r]^-{(\pi)_i} \ar[d]_{\delta_i}  & M_i \ar[d]^{\alpha_i}\\
X_4\ar[r]_-{(\pi)_4}& M_4
}
\]
for all $i=2.3$\;. Pulling back along $z_i$ we have
\[
\xymatrix{
X_i\times_{M_i} Z \ar[d]_{\pi_1}\ar[r]^-{\pi_2} & Z \ar[d]^-{z_i}\\
X_i\cong X_4\times_{M_4} M_i \ar[d]_{\pi_1}\ar[r]^-{\pi_2} & M_i\ar[d]^{\alpha_i}\\ 
X_4  \ar@{->>}[r]_-{(\pi)_4} & M_4\;,
}
\]
that is, $X_i\times_{M_i} Z \cong X_4\times_{M_4} Z$ for all $i=2,3,4$\;.  By Lemma~\ref{preliminaries, lemma:pullbackprincipal} there is a unique principal $G$-bundle over $Z$ (up to isormorphisms) $W$, such that, the maps of locales $z_i$ lift uniquely to maps of locales $x_i:W\to X_i$\;, that is, the following diagram
\[
\xymatrix{
W\ar[r]^-{x_i} \ar[d]_{(\pi)_W}  & X_i \ar[d]^{(\pi)_i}\\
Z\ar[r]_-{z_i}& M_i\;,
}
\]
is commutative and $\delta_2\circ x_2=\delta_3\circ x_3=x_4$ by \eqref{diag:pullbackorbitspaces}. Thus
\begin{equation*}
\xymatrix{
W \ar@/_/[ddr]_{x_3}\ar@/^/[drr]^{x_2}\ar@{.>}[dr]|{\exists! \psi}\\
&X_1\ar[r]^-{\gamma_2} \ar[d]_{\gamma_3}  & X_2 \ar[d]^{\delta_2}\\
&X_3\ar[r]_-{\delta_3}& X_4
}
\end{equation*}
because $X_1\cong X_3\times_{X_4} X_2$ by hypothesis. By Lemma~\ref{preliminaries, lemma: f/G}, $\psi:W\to Z$ induces a map of locales $z_1:Z\to M_1$ such that
\[
\xymatrix{
W\ar[r]^-{\psi}  \ar@{->>}[d]_{(\pi)_W}  & X_1  \ar@{->>}[d]^{(\pi)_1}\\
Z\ar[r]_-{z_1}& M_1
}
\]
is a pullback in \Loc, and $\beta_i\circ z_1=z_i$ holds for all $i=2,3$. It remains to show that $z_1$ is unique, in order to conclude that, notice that any map of locales $z'_1:Z\to M_1$ such that $\beta_i\circ z'_1=z_i$ for all $i=2,3$; extends to a map of locales $x'_1:\langle x_i,z'_1\rangle:W\to X_i\times_{M_i} M_i\cong X_1$ with $\gamma_i\circ x'_1=x_i$ for all $i=2,3$\;. This implies that $x'_1=\psi$ and finally that $z'_1=z_1$ as required. Conversely, let us suppose that the rightmost square is a pullback, \ie\ $M_1\cong M_3\times_{M_4} M_2$, we have to show that $X_1\cong X_3\times_{X_4} X_2$. Indeed, by Lemma~\ref{preliminaries, lemma:pullbackprincipal} we have $X_i\cong X_4\times_{M_4} M_i$ for all $i=1,2,3$. Thus,
\begin{align*}
X_3\times_{X_4} X_2 &\cong (X_4\times_{M_4} M_3)\times_{X_4} (X_4\times_{M_4} M_2)\\
&\cong X_4\times_{M_4}( M_3 \times_{M_4} M_2)\\
&\cong X_4\times_{M_4} M_1\\
&\cong X_1\,.
\end{align*}

\end{proof}

\paragraph{Principal bibundles.}

There are several notions of map from a groupoid $H$ to a groupoid $G$ which are based on bibundles between $G$ and $H$ and generalize continuous functors $\varphi:H\to G$. For instance, Hilsum and Skandalis~\cite{HS87} define a map $\varphi:W\to V$ between the spaces of leaves of two smooth foliated manifolds to be a \emph{principal $G$-$H$-bibundle} --- that is, a principal $G$-bundle over $H_0$ --- where $G$ and $H$ are the holonomy groupoids of $V$ and $W$, respectively. In the present paper, following~\cite{Mr99} and \cite{Funct2} , by a \emph{Hilsum--Skandalis map} from an open groupoid $H$ to an open groupoid $G$ will be meant the isomorphism class of a principal $G$-$H$-bibundle. Another name for such maps is \emph{bibundle functors} \cite{MeyerZhu}. The \emph{category $\HSG$ of Hilsum--Skandalis maps} is that whose objects are open groupoids and whose morphisms are the Hilsum--Skandalis maps, with the composition in $\HSG$ being induced by the tensor product of principal bibundles.

\begin{lemma}\label{compositionhsmaps}
Let $G,H$ and $K$ be open groupoids. Suppose that $X$ is a principal $G$-$H$-bibundle and $X'$ is a principal $H$-$K$-bibundle. Then $X\otimes_H X'$ is a principal $G$-$K$-bibundle.  
\end{lemma}
\begin{proof}
Let us prove that $X\otimes_H X'$ is a principal $G$-$K$-bibundle. In order to do that, notice that the mapping $\pi_{X'}:X\times_{H_0} X'\to X'$ is an open surjection because $\pi_X:X\to H_0$ is one. So, the induced map $\pi_{X'/H}: X\otimes_{H} X'\to X'/H\cong K_0$ is an open surjection by Lemma~\ref{preliminaries, lemma: f/G}. By Lemma~\ref{preliminaries, lemma:pullbackprincipal}, $X\times_{H_0} X'$ is a principal $G$-bundle over $X'$, because $X$ is a principal $G$-bundle over $H_0$, and
\[
\xymatrix{
X\times_{H_0} X'\ar[r]^-{\pi_2} \ar[d]_{\pi_1}  & X' \ar[d]^{p_{X'}}\\
X\ar@{->>}[r]_{\pi_{X}}& H_0\;,
}
\]  
is a pullback in \Loc\,. Therefore,
\[ 
G_1\times_{G_0} (X\times_{H_0} X')\cong (X\times_{H_0} X')\times_{X'}(X\times_{H_0} X')\,.    \]
Let us prove that $X\times_{H_0} X'$ is a principal $H$-bundle over $K_0$. In fact, since $X'$ is a principal $H$-bundle over $K_0$ we have
\begin{align*}
X\times_{H_0} H_1\times_{H_0} X' &\cong  X\times_{H_0} (H_1\times_{H_0} X') \\
&\cong X\times_{H_0} (X'\times_{K_0} X')\\
&\cong (X\times_{H_0} X')\times_{K_0} (X\times_{H_0} X')\,.
\end{align*}
Now, pulling back the quotient $\xymatrix{X\times_{H_0} X'\ar@{->>}[r] &(X\otimes_{H} X')}$ along the range map $r:G_1\to G_0$ yields a principal $H$-bundle over $G_1\times_{G_0}(X\otimes_{H} X')$, namely $G_1\times_{G_0}(X\times_{H_0} X')$ due to Lemma~\ref{preliminaries, lemma:pullbackprincipal}. Now,
let us consider the following squares:
\[
\xymatrix{
(X\times_{H_0} X')\times_{X'} (X\times_{H_0} X')\ar[r]^-{\pi_2} \ar[d]_{\pi_1}  & X\times_{H_0} X' \ar[d]^{\pi_{X'}}\\
X\times_{H_0} X'\ar[r]_-{\pi_{X'}}& X'\;
}
\]
\[
\xymatrix{
(X\otimes_{H} X')\times_{K_0} (X\otimes_{H} X')\ar[r]^-{\pi_2} \ar[d]_{\pi_1}  & X\otimes_{H} X' \ar[d]^{\pi_{X'/H}}\\
X\otimes_{H} X'\ar[r]_-{\pi_{X'/H}}& K_0\,.
}
\]
Notice that the upper square is a pullback of principal $H$-bundles over $K_0$ and $X\otimes_{H} X'$, respectively. Therefore by Lemma~\ref{preliminaries, lemma: twopullbacks}, the lower one is also a pullback of their orbit locales, which implies that $(X\times_{H_0} X')\times_{X'} (X\times_{H_0} X')$ is a principal $H$-bundle over $(X\otimes_{H} X')\times_{K_0} (X\otimes_{H} X')$. Finally, the following diagram:
\[
\vcenter{\xymatrix{
G_1\times_{G_0}(X\times_{H_0} X') \ar@{->>}[r] \ar[d]_-{\cong}  & G_1\times_{G_0}(X\otimes_{H} X') \ar[d]\\
(X\times_{H_0} X')\times_{X'} (X\times_{H_0} X') \ar@{->>}[r]& (X\otimes_{H} X')\times_{K_0} (X\otimes_{H} X')\,
}
}
\]
is a pullback in \Loc\, and
\[ 
G_1\times_{G_0} (X\otimes_{H} X')\cong (X\otimes_{H} X')\times_{K_0}(X\otimes_{H} X')
\]
due to Lemma~\ref{preliminaries, lemma: f/G}. This shows that $X\otimes_{H} X'$ is a principal $G$-$K$-bibundle.
\end{proof}

\section{Supported $Q$-modules}\label{section:supportedmodules}

The purpose of this section is to generalize for open groupoids the algebraic theory of supported modules of \cite{GSQS}.

\paragraph{$Q$-modules.} Let $Q$ be a groupoid quantale. By a \emph{(left) $Q$-module} is meant a locale $X$ equipped with a left $Q$-action  $(q,x)\mapsto q\cdot x$ together with a unital left $A$-module structure $(a,x)\mapsto {\lres{a}{x}}$, that besides the usual axioms of modules, it must satisfy the following additional conditions for all $q \in Q,a\in A$ and $x,y \in X$:
\begin{align}
 (\lres{a}{q})\cdot x &=  {\lres{a}{(q\cdot x)}}\label{modgq1}\\
 (\rres{q}{a})\cdot x &= q\cdot (\lres{a}{x})\label{modgq2}\\
 \lres{a}{(x\wedge y)} &= (\lres{a}{x})\wedge y\;.\label{modgq3} 
\end{align}
the condition \eqref{modgq3} implies, for all $a\in A$ and $x\in X$, that
\begin{equation}
\lres{a}{x}={\lres{a}{(1_X\wedge x)}}=(\lres{a}{1_X})\wedge x\;.
\end{equation}

\paragraph{Hilbert $Q$-modules.}
Let $Q$ be a groupoid quantale. By a \emph{pre-Hilbert $Q$-module} is meant a $Q$-module $X$ equipped with a binary operation:
\[   \langle -,- \rangle: X\times X \to Q    \]
called the \emph{inner product} satisfying for all $x,y\in X$, $a\in A$ and $q\in Q$ the following axioms:
\begin{align}
 \langle q\cdot x,y \rangle &=q\langle x,y \rangle\label{hilbertmod1}\\
 \lres{a}{\langle x,1_X \rangle} &=\langle \lres{a}{x},1_X \rangle\label{hilbertmod2} \\
  \langle \V_{\alpha}x_{\alpha},y \rangle &=\V_{\alpha}\langle x_{\alpha},y \rangle\label{hilbertmod3}\\
 \langle x,y \rangle &=\langle y,x \rangle^{\ast}\;.\label{hilbertmod4} 
\end{align}

\paragraph{Supported $Q$-modules.}
Let $Q$ be a groupoid quantale. By a \emph{supported $Q$-module} is meant a pre-Hilbert $Q$-module $X$ equipped with a monotone map:
\[   \spp_X: X\to A\;,  \]
satisfying for all $x,y\in X$ and $q\in Q$ the following conditions:
\begin{align}
\spp_X(1_X) &=1_A\label{sppmod1}\\  
\lres {\spp_X(x)}{1_X} &\leq \langle x,x\rangle\cdot 1_X\label{sppmod2}\\
\lres {\spp_X(x)}{x} &= x\;.\label{sppmod4}
\end{align}
Notice that combining \eqref{modgq3}, \eqref{sppmod2} and \eqref{sppmod4}, for all $x\in X$, we have
\begin{equation*}
x=\lres {\spp_X(x)}{x}=(\lres {\spp_X(x)}{1_X})\wedge x\leq \langle x,x\rangle\cdot 1_X\wedge x\leq x\;,
\end{equation*}
obtaining the following useful expression:
\begin{equation}\label{sppmod5}
x=\langle x,x\rangle\cdot 1_X\wedge x\;.
\end{equation}

\paragraph{Stably supported $Q$-modules.}
Let $Q$ be a groupoid quantale and $X$ be a supported $Q$-module. A support of $X$ is said to be \emph{stable} if the following condition holds:
\begin{equation}\label{stablespp}
\spp_X(q\cdot x)\leq \spp_Q(q)\quad \forall q\in Q,\forall x\in X\;.
\end{equation}

\begin{theorem}\label{theorem:stableinner1}
Let $Q$ be a groupoid quantale and $X$ be a stably supported $Q$-module. Then, for all $x\in X$, we have
\begin{equation}\label{sppformula}
\varsigma_X(x)=\varsigma_Q(\langle x,x \rangle)=\varsigma_Q(\langle x,1_X\rangle)\;.
\end{equation}
Moreover, $\spp_X$ is \emph{$A$-equivariant}, that is, for all $x\in X$ and $a\in A$ we have
\begin{equation}\label{sppequivariance}
\spp_X(\lres{a}{x})=a\wedge \spp_X(x)\;.
\end{equation}

\end{theorem}

\begin{proof}
The following sequence of (in)equalities will show \eqref{sppformula}:
\begin{align*}
\varsigma_Q(\langle x,x \rangle)&\leq \varsigma_Q(\langle x,1_X \rangle)\\
&= \varsigma_Q(\langle \lres{\varsigma_X(x)}{x},1_X  \rangle)& \text{[by \eqref{sppmod4}]}\\
&= \varsigma_Q( \lres{\varsigma_X(x)}{\langle x,1_X\rangle})& \text{[by \eqref{hilbertmod2}]}\\
&= \varsigma_X(x)\wedge \varsigma_Q(\langle x,y\rangle)& \text{($\spp_Q$ is equivariant)}\\
&\leq \varsigma_X(x)\\
&= \varsigma_X(\langle x,x \rangle\cdot 1_X\wedge x)& \text{[by \eqref{sppmod5}]} \\
& \leq \varsigma_X(\langle x,x \rangle\cdot 1_X)\\
& \leq \varsigma_Q(\langle x,x \rangle)& \text{($\spp_X$ is stable)}\\
& \leq \varsigma_Q(\langle x,1_X \rangle)\;.
\end{align*}
It remains to prove the equivariance of $\spp_X$. Indeed, for all $a\in A$ and $x\in X$, we have
 \begin{align*}
  \spp_X(\lres{a}{x}) & =\spp_Q(\langle \lres{a}{x},1_X \rangle) \\
  & = \spp_Q(\lres{a}{\langle x, 1_X \rangle})& \text{[by \eqref{hilbertmod2}]} \\
  & = a \wedge \spp_Q(\langle x,1_X \rangle)&  \text{($\varsigma_Q$ is equivariant)} \\
  & = a \wedge \spp_X(x)\;.
 \end{align*}

\end{proof}

\begin{corollary}\label{cor:restQ}
Let $Q$ be a groupoid quantale and $X$ be a stably supported $Q$-module. Then, for all $x\in X$ and $q\in Q$:
\begin{equation}
\lres{\spp_X(x)}{q}=\langle x,x\rangle1_Q\wedge q\;.
\end{equation}
\end{corollary}
\begin{proof}
Applying Theorem~\ref{theorem:stableinner1} we obtain
\begin{align*}
\lres{\spp_X(x)}{1_Q}&=\lres{\spp_Q(\langle x,x\rangle)}{1_Q}= \langle x,x\rangle1_Q\;.& \text{(by \cite[Lemma 4.1.11(6)]{QuijanoPhD})}
\end{align*}
Therefore, for all $q\in Q$ we have
\[ \lres{\spp_X(x)}{q}=(\lres{\spp_X(x)}{1_Q})\wedge q= \langle x,x\rangle1_Q\wedge q\;. \]
\end{proof}

We stress out that the existence of a support is a property of a pre-Hilbert $Q$-modules rather than extra structure.

\begin{lemma}\label{lemma:framehomo}
Let $Q$ be a groupoid quantale and $X$ be a stably supported $Q$-module. The mapping
$\lres{(-)}{1_X}:A\rightarrow X$ is the right adjoint of $\spp_X: X\rightarrow A$. In particular, the support of $X$ respects arbitrary joins and it is unique.
\end{lemma}
\begin{proof}
The unit of the adjunction follows from \eqref{sppmod4}, because $x={\lres{\varsigma(x)}{x}}\leq {\lres{\varsigma(x)}{1_X}}$, and the co-unit follows from the equivariance of $\spp_X$ (\cf, Theorem~\ref{theorem:stableinner1}),
\begin{align*}
\varsigma_X(\lres{a}{1_X}) &=a\wedge \varsigma_X(1_X)& \text{[by \eqref{sppequivariance}]}\\ 
&= a \wedge 1_A& \text{[by \eqref{sppmod1}]}\\ 
&= a\;.
\end{align*}
\end{proof}

\begin{lemma}
Let $Q$ be a groupoid quantale and $X$ be a stably supported $Q$-module. Let us assume that, for all $x\in X$ and $b\in A$, the following hold:
\begin{enumerate}
\item $\lres{b}{1_X}\leq \langle x,x \rangle\cdot 1_X$.
\item $\lres{b}{x}=x$,
\end{enumerate} 
Then, $b=\spp_X(x)$.
\end{lemma}
\begin{proof}
Using \eqref{sppequivariance} we obtain $b\wedge \spp_X(x)=\spp_X(\lres{b}{x})=\spp_X(x)$. Thus, $\varsigma_X(x) \leq b$. Finally, we have 
\begin{align*}
b &=b\wedge 1_A\\
&=b\wedge \varsigma_X(1_X)& \text{[by \eqref{sppmod1}]}\\
&=\varsigma_X(\lres{b}{1_X})& \text{($\spp_X$ is equivariant)}\\
&\leq \varsigma_X(\langle x,x \rangle\cdot 1_X)\\
&\leq \varsigma_Q(\langle x,x \rangle)& \text{($\spp_X$ is stable)}\\
&= \varsigma_X(x)\;.& \text{(by Theorem~\ref{theorem:stableinner1})}
\end{align*}    
\end{proof}

Finally, to close this section we shall provide some equivalent formulations of stability, which will be use later in this paper. 

\begin{lemma}\label{lemma:stablesupport}
Let $Q$ be a groupoid quantale and $X$ be a supported $Q$-module. Then, the following properties are equivalent, for all $x\in X$ and $q\in Q$:
\begin{align}
\varsigma_X(q\cdot x) &= \varsigma_Q(\rres{q}{\varsigma_X(x)})\label{stable1}\\
\varsigma_X(q\cdot x) &\leq \varsigma_Q(q)& \text{($\spp_X$ is stable)}\label{stable2}\\
\varsigma_X(q\cdot 1_X)&\leq \varsigma_Q(q)\;.\label{stable3}
\end{align}

\end{lemma}

\begin{proof}
The conditions \eqref{stable2} and \eqref{stable3} are clearly equivalent. Let us assume \eqref{stable1} and prove \eqref{stable2}. Indeed, for all $q\in Q$ and $x\in X$, we have $\varsigma_X(q\cdot x)= \varsigma_Q(\rres{q}{\varsigma_X(x)})\leq \varsigma_Q(q)$. Conversely, assume \eqref{stable2}, then
\begin{align*}
\varsigma_X(q\cdot x) &= \varsigma_X(q\cdot (\lres{\varsigma_X(x)}{x})) = \varsigma_X((\rres{q}{\varsigma_X(x)})\cdot x)& \text{[by \eqref{modgq2}]}\\
&\leq \varsigma_Q(\rres{q}{\varsigma_X(x)})\;.& \text{($\spp_Q$ is stable)}
\end{align*}
Applying the equivariance of $\varsigma_Q$ (and, in particular, the stability) the following sequence of (in)equalities will prove \eqref{stable1}:
\begin{align*}
\varsigma_X(q\cdot x) &\leq \varsigma_Q(\rres{q}{\varsigma_X(x)})=\varsigma_Q((\rres{q}{\varsigma_X(x)})1_Q)& \text{(by \cite[Proposition 4.1.15(2)]{QuijanoPhD})}\\
&=\varsigma_Q(q(\lres{\varsigma_X(x)}{1_Q}))\\
&=\varsigma_Q(q\langle x,x\rangle{1_Q})& \text{(by Corollary~\ref{cor:restQ})}\\
&=\varsigma_Q(q\langle x,x \rangle)& \text{($\spp_Q$ is stable)}\\
&= \varsigma_Q(\langle q\cdot x,x \rangle)& \text{[by \eqref{hilbertmod1}]}\\
&= \varsigma_Q(\langle \lres{\varsigma_X(q\cdot x)}{(q\cdot x)},x \rangle)& \text{[by \eqref{sppmod4}]}\\
&= \varsigma_Q(\langle( \lres{\varsigma_X(q\cdot x)}{q})\cdot x,x \rangle)& \text{[by \eqref{modgq1}]}\\
&= \varsigma_Q(( \lres{\varsigma_X(q\cdot x)}{q}) \langle x,x \rangle)& \text{[by \eqref{hilbertmod1}]}\\
&\leq \varsigma_Q(\lres{\varsigma_X(q\cdot x)}{q})\\
&= \varsigma_X(q\cdot x)\wedge \varsigma_Q(q)\\
&\leq \varsigma_X(q\cdot x)\;.
\end{align*}

\end{proof}

\section{Principal $Q$-locales}\label{section:principallocales}

\paragraph{$Q$-locales.}
Let $Q$ be a groupoid quantale. By a $Q$-locale will be meant a stably supported $Q$-module $X$ satisfying the following conditions:
\begin{description}
\item{\textbf{Q1:}}\label{multiplicativity} The mapping $\alpha_*: X\to Q\otimes_A X$ given by $\alpha_*(x)= \V_{q\cdot y\leq  x} q\otimes y$ preserves arbitrary joins [\cf\ \eqref{eq:rightadjointalpha}]. 
\item{\textbf{Q2:}}\label{unitlaws} $\V_{q\cdot y\leq x} (\lres{\upsilon(q)}{y})=x$\quad $\forall q\in Q$ and $\forall x,y\in X$.
\end{description}
The category of $Q$-locales, $Q$-\Loc, is that whose objects are the $Q$-locales and whose morphisms $f : Y\to X$ are frame homomorphism such that $f$ is a homomorphism of $Q$-modules and $A$-modules such that the following diagram commutes:
\[ \xymatrix{
A && X \ar_{\spp_X}[ll]\\
&& Y\ar^{\spp_Y}[ull]\ar_{f}[u],
}
\]
\ie\ the following condition holds:
\begin{equation}
\spp_X(f(y))=\spp_Y(y),  
\end{equation}

\begin{theorem}\label{theo:fullopenlocale}
Let $Q$ be a groupoid quantale and $X$ be $Q$-locale. Then $X$ is a fully open $G$-locale, where $G=\mathcal{G}(Q)$.
\end{theorem}

\begin{proof}
Since $X$ is a stably supported $Q$-module, let us put $p_!=\spp_X$, then by Lemma~\ref{lemma:framehomo} the mapping $p^*=\lres{(-)}{1_X}$ is the right adjoint of $p_!$ and therefore it is a frame homomorphism and the Frobenius condition is also satisfied, for all $x\in X$ and $a\in A$:
\begin{align*}
p_!(x\wedge p^*(a))&= \spp_X(x\wedge (\lres{a}{1_X}))\\
&= \spp_X(\lres{a}{x})& \text{(by \eqref{modgq3})}\\
&= a\wedge\spp_X(x)& \text{(by Theorem~\ref{theorem:stableinner1})}\\
&= a\wedge p_!(x)\;.
\end{align*}
Hence, we have an open map of locales $p:X\to G_0$ which is also surjective due to \eqref{sppmod1}.

Since the pullback of $r$ and $p$ is, in \Frm, the quotient of the frames coproduct $\opens(G_1)\otimes_{\opens(G_0)} X$ generated by the equalities:
\begin{equation}\label{pullback1}
\pi^*_1(r^*(a))=\pi^*_2(p^*(a))\;,\quad \text{for all $a\in A$,}
\end{equation}
if we stabilize (\ref{pullback1}) under finite meets, and noticing that $p^*(a)\wedge x={\lres{a}{x}}$ and $q\wedge r^*(a)=\rres{q}{a}$, it turns out that $G_1\times_{G_0} X$ coincides with the sup-lattice quotient generated by the equality $\rres{q}{a}\otimes x=q\otimes {\lres{a}{x}}$. Now, we can define an action $\act:G_1\times_{G_0} X\to X$ given by $\act^*=\alpha_*$ because the right adjoint $\alpha_*$ of the module action $\alpha:Q\otimes_A X\to X$ preserves arbitrary joins due to \textbf{Q1}. Let us prove now that $(X,\act,p)$ is a $G$-locale, in order to show that, we must verify the axioms of the definiton:
\begin{description}
\item{\textbf{Pullback}:} To show that
\[
\vcenter{\xymatrix{
 G_1\times_{G_0} X  \ar[d]_{\pi_1} \ar[r]^-{\act} & X \ar[d]^{p}  \\
 G_1  \ar[r]_-{d} & G_0 }
}
\]
is commutative, we shall show that
\begin{equation}\label{diag:directimagepullback}
\vcenter{\xymatrix{
 Q\otimes_A X  \ar[d]_{(\pi_1)_!} \ar[r]^-{\act_!=\alpha} & X \ar[d]^{p_!=\spp_X}  \\
 Q  \ar[r]_-{d_!=\spp_Q} & A, }
}
\end{equation}
is commutative. Firstly, recall from \cite[Lemma 3.1]{Funct2} that $(\pi_1)_!(q\otimes x)=\rres{q}{\spp_X(x)}$ for all $q\in Q$ and $x\in X$.
To show the diagram \eqref{diag:directimagepullback} is commutative, we apply Lemma~\ref{lemma:stablesupport}
\begin{equation*}
\vcenter{\xymatrix{
 q\otimes x  \ar@{|->}[d]_{(\pi_1)_!} \ar@{|->}[r]^-{\alpha} & q\cdot x \ar@{|->}[d]^{\spp_X}  \\
 \rres{q}{\spp_X(x)}  \ar@{|->}[r]_-{\spp_Q} & \spp_Q(\rres{q}{\spp_X(x)})=\spp_X(q\cdot x)}
}
\end{equation*}
\item{\textbf{Associativity}:} The associativity of $\act$ follows in a straightforward way
from the associativity of $\alpha$ because $ \alpha=\act_!$. (This is completely analogous to
the way in which the associativity of the multiplication of an open groupoid follows from the associativity of the multiplication of its quantale, (\cf\ \cite[Theorem 4.8]{Re07}).
\item{\textbf{Unitary}:} In order to prove this, we shall use the inverse image of $\act\circ \langle u\circ p, \ident_X \rangle$. In fact, for all $x\in X$, we have
\begin{align*}
([p^*\circ u^*,\ident_X]\circ \act^*)(x) & = [p^*\circ \upsilon,\ident_X](\V_{q\cdot y\leq x} q\otimes y)\\
& = \V_{q\cdot y\leq x} p^*(\upsilon(q)) \wedge y\\
& =  \V_{q\cdot y\leq x} (\lres{\upsilon(q)}{1_X}) \wedge y\\
& = \V_{q\cdot y\leq x} (\lres{\upsilon(q)}{y})\\
& = x& \text{(by \textbf{Q2})}
\end{align*}
Hence, $\act\circ \langle u\circ p, \ident_X \rangle=\ident_X$ as required.
\end{description}
Thus, $(X,\act,p)$ is a fully open $G$-locale.
\end{proof}

\begin{corollary}\label{cor:beck-chevalley}
Let $Q$ be a groupoid quantale and $X$ be a $Q$-locale. Then 
\begin{equation}
\lres{\spp_Q(q)}{z}=q\cdot 1_X\wedge z\;,
\end{equation} 
for all $q\in Q$.
\end{corollary}
\begin{proof} 
Let us put $G=\mathcal{G}(Q)$ the open groupoid of $Q$. By Theorem~\ref{theo:fullopenlocale} $X$ is a fully open $G$-locale. Therefore the diagram \[
\vcenter{\xymatrix{
 G_1\times_{G_0} X  \ar[d]_{\pi_1} \ar[r]^-{\act} & X \ar[d]^{p}  \\
 G_1  \ar[r]_-{d} & G_0 }
}
\] is a pullback. Then the Beck--Chevalley property yields
\[   \act_!\circ \pi_1^* = p^*\circ  d_!\;.      \]
Thus, for all $q\in Q$, we have
\[ q\cdot 1_X=\act_!(q\otimes 1_X)=(\act_!\circ\pi_1^*)(q)=(p^*\circ d_!)(q)= p^*(\spp_Q(q))={\lres{\spp_Q(q)}{1_X}}\;,   \]
which implies that for all $z\in X$:
\begin{align*}
\lres{\spp_Q(q)}{z} & = (\lres{\spp_Q(q)}{1_X}) \wedge z= q\cdot 1_X \wedge z\;.
\end{align*}
\end{proof}

\paragraph{Principal $Q$-locales.}

Let $Q$ be a groupoid quantale and $M$ be a locale. By a \emph{principal $Q$-locale over $M$} will be meant a $Q$-locale $X$ satisfying the following additional conditions:
\begin{description}
\item{\textbf{P1:}}\label{principalQbundle1}  \text{$X$ is an open $M$-locale with $\tspp(1_X)=1_M$}
\item{\textbf{P2:}}\label{principalQbundle2} $\tspp(q\cdot x)=\tspp(\lres{\spp_Q(q^*)}{x})$\quad \text{for all $q\in Q$ and $x\in X$}
\item{\textbf{P3:}}\label{principalQbundle3} The mapping $\varphi: X\otimes_M X\to  Q\otimes_{A} X$ given by 
\[ \varphi(x\otimes y)=\V_{q\cdot z\leq x} q\otimes (z\wedge y) \] 
is a frame isomorphism.
\end{description}

\begin{lemma}\label{lemma:principalQimpliesprincipalG}
Let $Q$ be a groupoid quantale and $X$ be a principal $Q$-bundle over $M$. Then $X$ is a fully open principal $G$-bundle over $M$, where $G=\mathcal{G}(Q)$.
\end{lemma}
\begin{proof}
Due to Theorem~\ref{theo:fullopenlocale}, it remains to prove that $X$ is a principal $G$-bundle over $M$. The condition~\textbf{P1} implies that there exists an open surjective map of locales $\pi: X\to M$. The condition~\textbf{P2} guarantees the commutativity of \eqref{diag:invarianceofaction}, because the diagram 
\[
\vcenter{\xymatrix{
Q\otimes_{A} X\ar[r]^-{\act_!=\alpha} \ar[d]_{(\pi_2)_!}  &X\ar[d]^{(\pi)_!=\tspp}\\
X\ar[r]_{(\pi)_!=\tspp} & 
}
}
\]
commutes in $\SL$. Indeed, for all $q\in Q$ and $x\in X$ we have 
\[
\vcenter{\xymatrix{
q\otimes x\ar@{|->}[r]^-{\act_!=\alpha} \ar@{|->}[d]_{(\pi_2)_!}  &q\cdot x\ar@{|->}[d]^{(\pi)_!=\tspp}\\
\lres{\spp_Q(q^*)}{x}\ar@{|->}[r]_-{(\pi)_!=\tspp} & \tspp(\lres{\spp_Q(q^*)}{x})=\tspp(q\cdot x)\;,
}
}
\]
where $(\pi_2)_!$ is the direct image homomorphism of the map $\pi_2:G_1\times_{G_0} X\to X$\;, which by \cite[Lemma 3.1]{Funct2}, $(\pi_2)_!$ is given, for all $q\in Q$ and $x\in X$, by $(\pi_2)_!(q\otimes x)={\lres{\spp_Q(q^*)}{x}}$.
Finally, the mapping $\varphi: X\otimes_M X\to  Q\otimes_{A} X$ can be written as: 
\begin{align*}
\varphi(x\otimes y) &=\V_{q\cdot z\leq x} q\otimes (z\wedge y)\\
&= \V_{q\cdot z\leq x} (q\otimes z) \wedge (1_Q\otimes y)\\
&= \act^*(x) \wedge \pi^*_2(y)\\
&= [\act^*,\pi^*_2](x\otimes y)\;,
\end{align*}
which is a frame isomorphism by the condition~\textbf{P3}, the latter implies that the pairing map $\langle \act,\pi_2 \rangle$ establishes the isomorphism between $G_1\times_{G_0} X$ and $X\times_M X$ in \Loc.

\end{proof}

\begin{theorem}\label{theo:principalGimpliesprincipalQ}
Let $G$ be an open groupoid and $X$ be a fully open principal $G$-bundle over $M$. Then $X$ is a principal $Q$-locale over $M$, where $Q=\opens(G)$.
\end{theorem}
\begin{proof}
Since $X$ is a $G$-locale, the Beck--Chevalley condition implies 
\[  \act_!\circ\pi^*_1=p^*\circ d_!\;,  \]
therefore, for all $q\in Q$, we have
\begin{equation}\label{eq:beckchevalley1}
q\cdot 1_X = \act_!(q\otimes 1_X) = \act_!(\pi^*_1(q))=p^*(\spp_Q(q))={\lres{\spp_Q(q)}{1_X}}\;,
\end{equation}
where $p^*={\lres{(-)}{1_X}}$\;.
Let us define the following  module structures, for all $q\in Q$, $a\in A$ and $x\in X$;
\begin{align*}
(q,x)&\mapsto q\cdot x:=\act_!(q\otimes x)\\
(a,x)&\mapsto {\lres{a}{x}}:=p^*(a)\wedge x\;.
\end{align*} 
Let us prove that $X$ is indeed a left $Q$-module, by showing that the following axioms hold, for all $q\in Q$, $a\in A$ and $x\in X$:
\begin{description}
\item{\eqref{modgq1}:} $(\lres{a}{q})\cdot x=(d^*(a)\wedge q)\cdot x=\act_!((d^*(a)\wedge q)\otimes x)=\act_!(\act^*(p^*(a))\wedge (q\otimes x))$, the latter equation holds because the diagram~\textbf{Pullback} is commutative, so we have $\pi^*_1(d^*(a))=\act^*(p^*(a))$, if we stabilize under finite meets, we get 
$(d^*(a)\wedge q)\otimes x=\act^*(p^*(a))\wedge (q\otimes x)$. Finally applying the Frobenius condition of $\act$ we obtain $\act_!(\act^*(p^*(a))\wedge (q\otimes x))= \act_!(q\otimes x)\wedge p^*(a)= q\cdot x\wedge p^*(a)={\lres{a}{(q\cdot x)}}$ as required.
\item{\eqref{modgq2}:} $(\rres{q}{a})\cdot x=(q\wedge r^*(a))\cdot x=\act_!((q\wedge r^*(a))\otimes x)=\act_!(q\otimes (x\wedge p^*(a)))$, the latter is obtained by stabilizing under finite meets the equation $\pi^*_1(r^*(a))=\pi^*_2(p^*(a))$. Now $\act_!(q\otimes (x\wedge p^*(a)))=\act_!(q\otimes (\lres{a}{x}))=q\cdot (\lres{a}{x})$ as we wanted.  
\item{\eqref{modgq3}:} It is trivial. 
\end{description}   
Now since $X$ is a principal $G$-locale over $M$, in particular the pairing map $\langle \act, \pi_2\rangle:G_1\times_{G_0} X\to X\times_M X$ is an isomorphism  in \Loc, there exists a map of locales $\theta: X\times_M X\to G_1$ such that the following diagram is a pullback in \Loc:
\begin{equation}\label{diag:principality2}
\vcenter{\xymatrix{
 X\times_{M} X  \ar[d]_{\theta} \ar[r]^-{\pi_2} & X \ar[d]^{p}  \\
 G_1  \ar[r]_-{d} & G_0 }
}
\end{equation}
this implies that $\theta$ is open because so is $p$ (obviously $\pi_2$ is open too). Let us prove that $\theta_!$ is an inner product. In fact, for all $x,y\in X$, $q\in Q$ and $a\in A$:
\begin{description}
\item{\eqref{hilbertmod1} and \eqref{hilbertmod4}:} follow from Lemma~\ref{localicreasoningwithglobalpoints}, and combining \eqref{hilbertmod1} and \eqref{hilbertmod4} we have 
\begin{equation}\label{hilbertmod5}
\theta_!(x\otimes y)q=\theta_!(x\otimes q^*x)\quad \text{for all $x,y\in X$ and $q\in Q$}\;.
\end{equation}
\item{\eqref{hilbertmod3}:} It obvious.
\item{\eqref{hilbertmod2}:} Applying the direct images of the diagram~\eqref{diag:principality2}, we have that the diagram
\begin{equation*}\label{diag:principality3}
\vcenter{\xymatrix{
 X\otimes_{M} X  \ar[d]_{\theta_!} \ar[r]^-{(\pi_2)_!} & X \ar[d]^{p_!}  \\
 Q  \ar[r]_-{d_!=\spp_Q} & A}
}
\end{equation*}
is commutative in $\SL$, which implies that for all $x,y\in X$:
\begin{align*} 
\spp_Q(\theta_!(x\otimes y))&=p_!((\pi_2)_!(x\otimes y))\\
&= p_!(\lres{p_!(x)}{y})&\text{(by \cite[Lemma 3.1]{Funct2})}\\
&= p_!(p^*(p_!(x))\wedge y)\\
&=p_!(x)\wedge p_!(y)\;.&\text{[by \eqref{preliminaries, eq: frobeniuscondition} of $p$]}
\end{align*}
In particular, taking into account that $p_!(1_X)=1_A$, for all $x\in X$, we have
\begin{equation}\label{eq:theta_!-1}
\spp_Q(\theta_!(x\otimes1_X))=p_!(x)\wedge p_!(1_X)=p_!(x)\wedge 1_A
=p_!(x)=\spp_Q(\theta_!(x\otimes x))\;.
\end{equation}
Then
\begin{align*}
\spp_Q(\theta_!(\lres{a}{x}\otimes 1_X))&=p_!(\lres{a}{x})=p_!(x\wedge p^*(a))\\
&=a\wedge p_!(x)&\text{[by \eqref{preliminaries, eq: frobeniuscondition} of $p$]}\\
&=a\wedge \spp_Q(\theta_!(x\otimes1_X))\\ 
&=\spp_Q(\lres{a}{\theta_!(x\otimes 1_X)})\;,&\text{($\spp_Q$ is equivariant)}
\end{align*}
which implies $\lres{\spp_Q(\theta_!(\lres{a}{x}\otimes 1_X))}{1_Q}={\lres{\spp_Q(\lres{a}{\theta_!(x\otimes 1_X)})}{1_Q}}$\;. Finally
\begin{align*}
\theta_!(\lres{a}{x}\otimes 1_X)&=\theta_!(\lres{a}{x}\otimes {\lres{1_A}{1_X}})\\
&=\theta_!(\lres{a}{x}\otimes {\lres{\spp_Q(1_Q)}{1_X}})&\text{[by \eqref{sppmod1}]}\\
&=\theta_!(\lres{a}{x}\otimes 1_Q1_X)&\text{[by \eqref{eq:beckchevalley1}]}\\
&=\theta_!(\lres{a}{x}\otimes 1^*_Q1_X)\\
&=\theta_!(\lres{a}{x}\otimes 1_X)1_Q&\text{[by \eqref{hilbertmod5}]}\\
&={\lres{\spp_Q(\theta_!(\lres{a}{x}\otimes 1_X))}{1_Q}}\\
&={\lres{\spp_Q(\lres{a}{\theta_!(x\otimes 1_X)})}{1_Q}}\\
&={\lres{a}{\theta_!(x\otimes 1_X)1_Q}}\\
&={\lres{a}{\theta_!(x\otimes 1_X)}}\;.& \text{[by \eqref{eq:beckchevalley1} and \eqref{hilbertmod5}]}
\end{align*}
This proves that $\theta_!$ is an inner product and from now on we will denoted it by $\langle -,-\rangle$ and thus $X$ is a pre-Hilbert $Q$-module. 
\end{description}
Since $p$ is an open surjection, its direct image $p_!:X\to A$ is a monotone map satisfying for all $x\in X$:
\begin{description}
\item{\eqref{sppmod1}:} $p_!(1_X)=1_A$ because $p$ is surjective.
\item{\eqref{sppmod2}:} using \eqref{eq:theta_!-1} and \eqref{eq:beckchevalley1} we have
\begin{align*}
 \lres{p_!(x)}{1_X} &={\lres{\spp_Q(\langle x,x\rangle)}{1_X}}\\
 &={\lres{\spp_Q(\langle x,x\rangle)}{(1_Q\cdot 1_X)}}\\
 &=({\lres{\spp_Q(\langle x,x\rangle)}{1_Q}})\cdot 1_X\\
 &=(\langle x,x\rangle 1_Q)\cdot 1_X\\
 &=\langle x,x\rangle(1_Q\cdot 1_X)\\
 &=\langle x,x\rangle\cdot 1_X\;.
 \end{align*} 
\item{\eqref{sppmod4}:} $\lres{p_!(x)}{x}=p^*(p_!(x))\wedge x =x$ is due to the unit of the adjunction $p^*\circ p_!\geq \ident$. 
\end{description}   
The Frobenius condition of $p$ give us the stability of $p_!$, as follows:
\begin{align*} 
p_!(q\cdot x)=p_!((\lres{\spp_Q(q)}{q})\cdot x)=p_!(\lres{\spp_Q(q)}{(q\cdot x)})&=p_!(p^*(\spp_Q(q))\wedge (q\cdot x))\\
&=\spp_Q(q)\wedge(q\cdot x)\le \spp_Q(q)\;.  
\end{align*}
Therefore $X$ is a stably supported $Q$-module with support $p_!=\spp_X$. Now let us prove that $X$ is a $Q$-locale. In order to see this, it suffices to show the following two axioms:
\begin{description}
\item{\textbf{Q1:}} since $\act$ is an open map of locales, $\act_!$ exists and $\act^*:X\to Q\otimes_A X$ giving by 
\[ \act^*(x)=\V_{q\cdot z\leq x} q\otimes z\;    \]
is a frame homomorphism. 
\item{\textbf{Q2:}} a direct translation of the diagram~\textbf{Unitary} in terms of inverse images yields $\langle u\circ p,\ident_X \rangle^*\circ \act^*=\ident_X$. Then, for all $x\in X$, we have
\begin{align*}
x&=\langle u\circ p,\ident_X \rangle^*(\act^*(x))=[p^*\circ \upsilon, \ident_X](\V_{q\cdot y\leq x}q\otimes y)\\
&=\V_{q\cdot y\leq x}p^*(\upsilon(q))\wedge y=\V_{q\cdot y\leq x}(\lres{\upsilon(q)}{y})\;.
\end{align*}
Hence, $X$ is a $Q$-locale.
 \end{description}
Finally, let us prove that $X$ is a principal $Q$-locale. \textbf{P1} follows from the fact that $\pi$ is an open surjection, then $X$ is an open $M$-locale with support $\pi_!=\tspp$ satisfying $\tspp(1_X)=1_M$. \textbf{P2} holds because the diagram~\eqref{diag:invarianceofaction} is commutative, so we have $\pi_!\circ\act_!=\pi_!\circ(\pi_2)_!$\;, then
\[  \tspp(q\cdot x)=\tspp(\act_!(q\otimes x))=\tspp((\pi_2)_!(q\otimes x))=\tspp(\lres{\spp_Q(q^*)}{x})\;.   \]
\textbf{P3} is satisfied because the pairing map $\langle \act, \pi_2\rangle$ is an isomorphism between  $G_1\times_{G_0} X$ and $X\times_M X$ in \Loc. Therefore the co-pairing map $[\act^*,\pi^*_2]$ is a frame isomorphism, moreover:
\begin{align*}
[\act^*,\pi^*_2](x\otimes y)&= \act^*(x) \wedge \pi^*_2(y)\\
&= \V_{q\cdot z\leq x} (q\otimes z) \wedge (1_Q\otimes y)\\
&=\V_{q\cdot z\leq x} q\otimes (z\wedge y)\\
&= \varphi(x\otimes y)\;,
\end{align*}
Hence, the mapping $\varphi: X\otimes_M X\to  Q\otimes_{A} X$ is a frame isomorphism. 
\end{proof}

Thus we conclude that the fully open principal $G$-bundles and the principal $\opens(G)$-locales, for any open groupoid $G$, amount to the same thing.

\begin{corollary}\label{cor:bijectionprincipalbundles}
Let G be an open groupoid. The assignment $X\mapsto \opens(X)$ from fully open principal $G$-bundles to principal $\opens(G)$-locales is a bijection.
\end{corollary}
\begin{proof}
This follows from Lemma~\ref{lemma:principalQimpliesprincipalG} and Theorem~\ref{theo:principalGimpliesprincipalQ}.
\end{proof}

\vspace*{5mm}
\noindent {\sc
{\it E-mail:} {\sf quijano.juanpablo@gmail.com}

\end{document}